\documentclass[12pt]{amsart}
\usepackage{fancyhdr}
\usepackage{txfonts}

\usepackage{amsmath,amsthm,amssymb,hyperref, mdframed}
\usepackage{mathrsfs} 
\usepackage{tikz}
\usepackage[all]{xy}
\usetikzlibrary{decorations.markings}
\usetikzlibrary{backgrounds,shapes}
\usetikzlibrary{patterns,hobby}
        \usepackage{pgfplots}
        \pgfplotsset{compat=1.6}

\usepackage{subfig}

\usepackage{color}
\usepackage{hyperref}
\hypersetup{
    colorlinks=true, 
    linktoc=all,     
    linkcolor=blue,  
    citecolor=blue,
    filecolor=blue,
    urlcolor=blue
}

\definecolor{aliceblue}{rgb}{0.9, 0.95, 1.0}
\definecolor{pallido}{RGB}{221,227,227}

\usepackage{geometry} 

\theoremstyle{plain}                    
\newtheorem{thm}{Theorem}[section]

\newtheorem{lem}[thm]{Lemma}

\newtheorem{defn}[thm]{Definition}

\newtheorem{prop}[thm]{Proposition}
\newtheorem{cor}[thm]{Corollary}

\theoremstyle{definition}

\newtheorem{ques}[thm]{Question}

\theoremstyle{remark}

\title[From discrete to dense: explorations in the moduli space of triangles]{From discrete to dense:\\ explorations in the moduli space of triangles}

\begin{document}

\author{Aahana Aggarwal \and Subhojoy Gupta \and Ajay K. Nair}

\address{Department of Mathematics, Indian Institute of Science, Bangalore, India}
\email{subhojoy@iisc.ac.in}
\email{ajaynair.mv@gmail.com}
\address{Independent Researcher} 
\email{aggarwalaahana1@gmail.com}

\date{}

\begin{abstract}
The moduli space of triangles is a two-dimensional space that records triangle shapes in the plane, considered up to similarity. We study the subset corresponding to \textit{lattice triangles}, which are triangles whose vertices have integer coordinates. We prove that this subset is \textit{dense}, that is, every triangle shape can be approximated arbitrarily well by lattice triangles. However, when one restricts to lattice triangles in the square $[-N,N]^2$, their shapes do \textit{not} become uniformly distributed in the moduli space as $N$ grows. Along the way, we encounter connections with geometry, number theory, analysis, and probability.
 \end{abstract}

\maketitle

\section{Introduction}

A \textit{lattice triangle} is a triangle in the plane whose vertices lie in the integer lattice  $\mathbb{Z}^2 = \{ (m,n) \vert\ m,n\in \mathbb{Z}\}$.  It is an elementary yet non-trivial fact that a lattice triangle cannot be equilateral; see \S \ref{subsection:LatticeTrianglesAreNotEquilateral} for a proof of that fact, and \cite{Shirali} for four other proofs,  involving basic geometry and number theory. 

\medskip

In this article, we shall address the broader geometric question --  what do lattice triangles look like when considered up to similarity? That is, if we forget their actual size and position in the plane, how are their \textit{shapes} distributed amongst all possible triangle shapes?  To make this more precise, we introduce the \textit{moduli space of triangles} $\mathcal{M}_\Delta$,  which is the parameter space recording triangles up to translation, rotation, reflection, and rescaling. Any such similarity class of a triangle is uniquely determined by its three side-lengths, which can be further rescaled to have a fixed perimeter; thus, the moduli space $\mathcal{M}_\Delta$ is two-dimensional.  We give details of this in \S\ref{subsection:ModuliSpaceOfTriangles}, and describe how to identify this as a domain in the plane using these normalized side-lengths.


\medskip

The similarity class of any lattice triangle uniquely determines a point in this planar domain representing the moduli space $\mathcal{M}_\Delta$. 
Our first result is about the set of points determined by \textit{all} lattice triangles. 

\begin{thm}\label{thm:main}
    The set of all lattice triangles $\mathcal{L}$ determines a dense set in $\mathcal{M}_\Delta$. 
\end{thm}

In other words, the similarity class of \textit{any} triangle can be approximated arbitrarily well by lattice triangles.
Our proof of this involves Dirichlet's approximation theorem, a basic result in analysis that can be thought of as the first fundamental result in the field of Diophantine approximation. In particular, we shall need a two-variable version of Dirichlet's theorem, that we prove in \S \ref{subsection:SimultaneousDirichletApproximation}
using only the pigeonhole principle. 

\medskip

We then turn to study the finer \textit{distribution} of the set in moduli space determined by $\mathcal{L}$.  Our question is the following: if we look at lattice triangles with vertices inside the square $[-N,N]\times[-N,N]$, and record their shapes in the moduli space $\mathcal{M}_\Delta$, do these points become evenly spread out as $N$ grows? In the language of measure theory, this asks whether the corresponding subsets become \textit{equidistributed} in the moduli space, as $N\to \infty$. To frame this more precisely, we define the notion of ``equidistribution" of ``weighted sets" (where points are counted with multiplicity) in \S \ref{subsection:TheQuestionOfEquidistribution} -- see Question \ref{ques:equi} for the final statement of the question. 

We give a negative answer to this question in our next result. 

\begin{thm}\label{thm:nequi}
The similarity classes of lattice triangles in $[-N,N]\times [-N,N]$, counted with multiplicity, do not equidistribute, that is, the corresponding sequence of distributions does not converge to the uniform (area) measure on $\mathcal{M}_\Delta$, as $N\to \infty$. 
\end{thm} 

We were led to this result by computer experiments (see Figure \ref{fig:plot}).  It is surprising to us that although the limiting distribution is not uniform, it is \textit{close} to being so; we do not know a conceptual reason why this has to be the case.  It also remains to determine the exact limiting distribution function in $\mathcal{M}_\Delta$ -- see Question \ref{ques:dist}. The beautiful paper \cite{ES} proves an equidistribution result, using an alternative parametrization of $\mathcal{M}_\Delta$, and a different way of sampling points; it would be interesting to understand how it relates to this work. 

\medskip

In the final section, we shall conclude with some other open questions and directions for the enterprising reader to pursue.

\bigskip

\textbf{Acknowledgements.} This work was done as part of the RSI-India program for high-school students; we are grateful to the organizers. AA is grateful to SG and AKN for their guidance and support.  We thank Konstantin Delchev, who pointed out Langford's work, and Manjunath Krishnapur for helpful comments. We are grateful to the editors and anonymous referees, who pointed out an important correction in a previous version, and made numerous suggestions to improve the article.  We acknowledge the support of the Department of Science and Technology, Govt. of India grant no. CRG/2022/001822. There are no relevant financial or non-financial competing interests to report.

\bigskip

\section{Preliminaries}

\subsection{Lattice triangles are not equilateral}\label{subsection:LatticeTrianglesAreNotEquilateral}

As mentioned in the Introduction, although the following result is well-known, the following elementary proof is (to the best of our knowledge) new.

\begin{lem}\label{equi}
    If $\Delta$ is a lattice triangle, then $\Delta$ is not equilateral.
\end{lem}
\begin{proof}
    Assume otherwise, i.e.,\ $\Delta$ is an equilateral lattice triangle.  By a translation, we can assume that one of the vertices is the origin. Namely, let the vertices of $\Delta$ be
\[
A = (0, 0), \quad B = (x, y), \quad C = (m, n),
\]
where all coordinates \(x,y,m,n \in \mathbb{Z} \) and $A,B,C$ are distinct. 


Since $ABC$ is an equilateral triangle, 
the third vertex \( C \) must be obtained by rotating \( \overrightarrow{AB} \) by \( 60^\circ \) about point \( A \) (see Figure \ref{fig:lattice}). Since \( \overrightarrow{AB} = (x, y) \) this rotation yields
\[
C = (x, y)
\begin{pmatrix}
\cos60^\circ & -\sin60^\circ \\
\sin60^\circ & \cos60^\circ
\end{pmatrix}^T
=
\left(
\frac{1}{2}x - \frac{\sqrt{3}}{2}y,\ 
\frac{\sqrt{3}}{2}x + \frac{1}{2}y
\right).
\]

\begin{figure}
\centering
\includegraphics{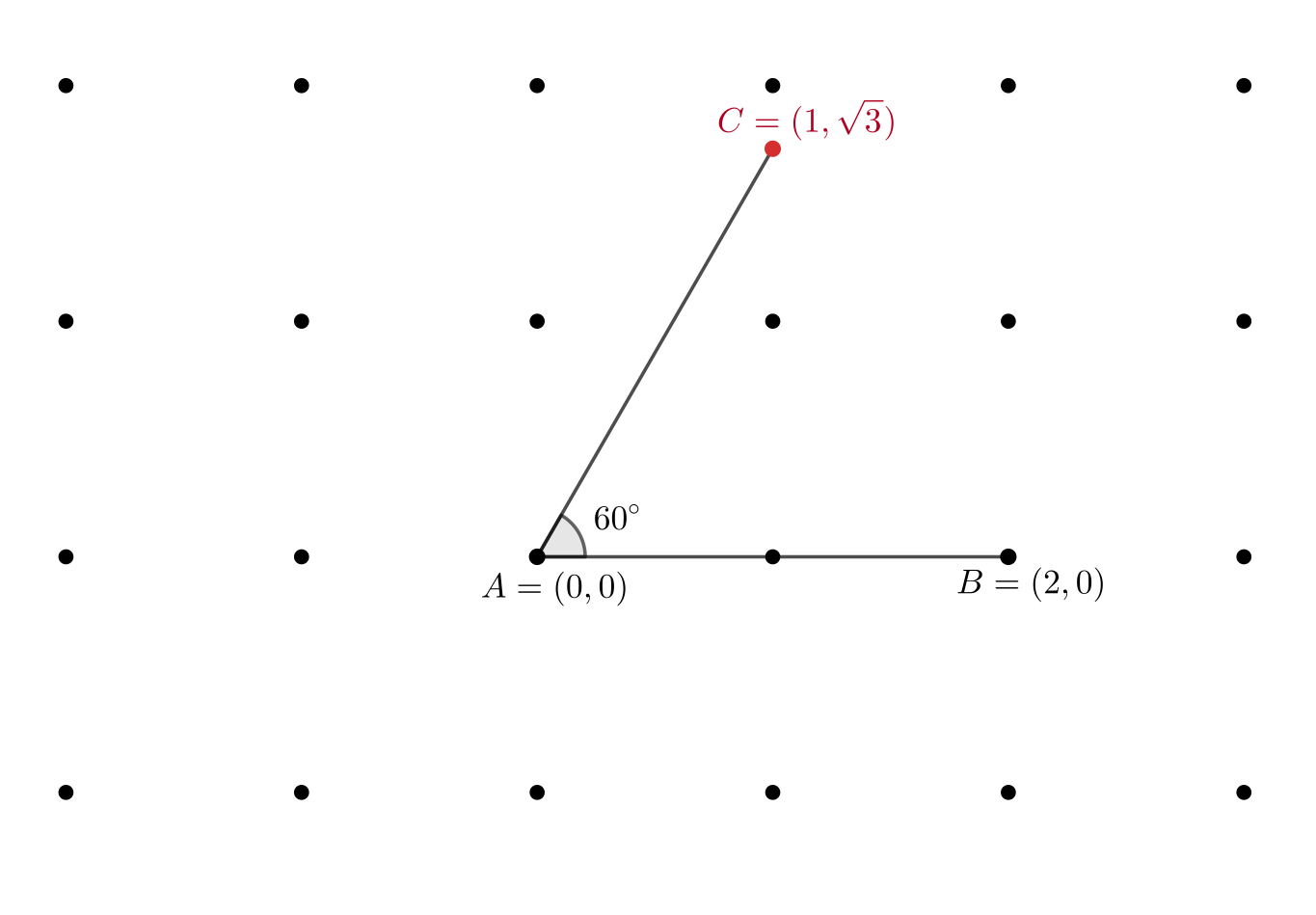}
\caption{Vector $\overrightarrow{AC}$ is obtained by rotating $\overrightarrow{AB}$ counterclockwise by $60^\circ$; if $A,B$ are lattice points, then $C$ is not.  }
\label{fig:lattice}
\end{figure}

Now since $\sqrt{3}$ is irrational, and \( x, y \) are integers, the coordinates of \( C \) will be irrational unless \( x = y = 0 \). However, \( x = y = 0 \) implies \( A = B \), which contradicts our assumption that the points are distinct. Therefore, \( C \) cannot lie in \( \mathbb{Z}^2 \), and our proof is complete. 
\end{proof}

\noindent \textit{Remark.} In fact, one can show that the only regular polygon with vertices in the integer lattice, is the square. (See, for example,  \cite{HamkinsBlog}.) One characterization  of lattice triangles is that they are exactly the triangles for which $\tan(\alpha)$ is  \textit{rational} for each interior angle $\alpha$ (see \cite{Beeson}).

\subsection{Moduli space of triangles}
\label{subsection:ModuliSpaceOfTriangles}
We say two triangles $T_1$ and $T_2$ in the Euclidean plane $\mathbb{R}^2$ are \textit{similar} if one can be transformed into the other by similarity, i.e.,\ a rigid motion (a composition of a translation,  rotation and/or reflection) composed by a scaling. The moduli space of triangles $\mathcal{M}_\Delta$  is the set of equivalence classes of triangles for the equivalence relation $\sim$, where $T_1 \sim T_2$ if there is a similarity that takes $T_1$ to $T_2$. In this section we shall describe how to identify this moduli space with a region  on the plane (which is also a triangle!) -- see Figure \ref{fig:IsoscelesAndEquilateral}.

\medskip

 We shall be using the side-lengths of a triangle to pin down its equivalence class; for this we need the following fact in Euclidean geometry which is well-known (and we omit a proof).

\begin{lem}\label{euc}
\begin{itemize}
\item[(a)] The side-lengths $a,b,c \in \mathbb{R}_+$ of a triangle $T$ in $\mathbb{R}^2$ satisfy the triangle inequalities: $a+b>c$, $b+c>a$ and $c+a >b$. 


\item[(b)] Two triangles $T_1$ and $T_2$ in $\mathbb{R}^2$ are similar if and only if their side-lengths are proportional with respect to the same constant, namely if the side-lengths of one is $(a,b,c)$, then the other has side-lengths $(\lambda a, \lambda b, \lambda c)$ for some $\lambda >0$. 
\end{itemize} 

\end{lem} 

\medskip






We first introduce another space of triangles with \textit{labelled} edges, that we call the \emph{Teichm\"{u}ller space of triangles}, denoted by  $\mathcal{T}_\Delta$. More formally,  $\mathcal{T}_\Delta$ is the set of triangles in the plane with sides labelled (by, say $\{1,2,3\}$) up to the following equivalence: $T_1 \sim T_2$ if there is a similarity that takes $T_1$ to $T_2$ and preserves the labels on sides (i.e., takes edge-$i$ of $T_1$ to edge-$i$ of $T_2$, for each $i=1,2,3$).

Note that $\mathcal{T}_\Delta$ is a bigger space than $\mathcal{M}_\Delta$, in the sense that there is a surjective map $\pi:\mathcal{T}_\Delta \to \mathcal{M}_\Delta$  obtained by forgetting the labelling. For instance, an isosceles triangle has a reflection that is a self-map interchanging the two equal sides; as labelled triangles, these would be different points in $\mathcal{T}_\Delta$, but these would project to the same point in $\mathcal{M}_\Delta$.  These terms are inspired by the moduli space and Teichm\"{u}ller space of Riemann surfaces (see \cite{AcuteT} for a recent paper that explores this analogy). 

\medskip
In the Lemma below we use the three side-lengths (and Lemma \ref{euc}) to describe the latter  space $\mathcal{T}_\Delta$  in terms of three real parameters with a cyclic symmetry. (Here, although we say three parameters, we normalize their sum to equal $2$ by a rescaling, that makes the space two-dimensional.)  The moduli space $\mathcal{M}_\Delta$ can then be described as a quotient of $\mathcal{T}_\Delta$  by the symmetry (of relabelling the sides) -- see Corollary \ref{cor:param}. (See \cite[\S1.9]{Behrend}, or \cite{Stewart} for an equivalent account.)

These  parametrizations equip each of these spaces with a topology, namely two points in $\mathcal{T}_\Delta$ (or $\mathcal{M}_\Delta$) are ``close" if their corresponding parameters are close (as real numbers). This is essential as the notion of ``dense" in the statement of Theorem \ref{thm:main} is with respect to this topology. 

\medskip

\begin{lem}\label{param} We have the parametrization
    $$\mathcal{T}_\Delta \cong\{ (a,b,c) \in \mathbb{R}_+^3\ \vert\ \ a, b, c<1 \text{ and } a+b+c =2\}.$$
    
In other words, any triple of positive real numbers in the set above corresponds to a unique point in $\mathcal{T}_\Delta$, and vice versa. 
\end{lem}

\begin{proof}
    Let $(a,b,c)$ be such that $ a,b,c<1 \text{ and } a+b+c =2$. 
    We need to prove that there exists a triangle with $(a,b,c)$ as side lengths. It is enough to prove that the tuple $(a,b,c)$ satisfies the triangle inequalities.

    Now, since $a+b+c = 2$, we derive 
    $$ a + b =  a+ (2 - c - a)  =  2 - c > c$$
   since $ c<1$.
   
    By the cyclic symmetry of $a,b$ and $c$, the same argument works for the remaining triangle inequalities. This implies that  $$\{ (a,b,c) \in \mathbb{R}_+^3\ \vert\ \ a, b, c<1 \text{ and } a+b+c =2\} \subseteq \mathcal{T}_\Delta.$$

    To prove the other direction, consider a triangle $T$ with side lengths $a, b ,c$. By scaling if needed, we can assume that $a+b+c=2$. Using the triangle inequalities, we will prove that $a,b,c < 1$. Observe that
    \begin{align*}
        a + b &> c  \implies a + (2- c -a) > c  \implies 2 - c > c  \implies c < 1
    \end{align*}

    and as before, by symmetry,  the same argument yields $a<1$ and $b<1$. This implies that 
     $$\mathcal{T}_\Delta \subseteq \{ (a,b,c) \in \mathbb{R}_+^3\ \vert\ \ a, b,  c<1 \text{ and } a+b+c =2\}$$
     and we are done. \end{proof}

\begin{cor}\label{cor:param}  
The moduli space of triangles $\mathcal{M}_\Delta$ is the quotient space $\mathcal{T}_\Delta/S_3$ where $S_3$ is the group of permutations of a $3$-element set, and we have the parametrization 
       $$\mathcal{M}_\Delta \cong \{ (a,b,c) \in \mathbb{R}_+^3\ \vert\ \ a\leq b\leq c<1 \text{ and } a+b+c =2\}.$$
\end{cor}
\begin{proof} Let $\pi:\mathcal{T}_\Delta \to \mathcal{M}_\Delta$ be the projection map obtained by forgetting the labelling. This is the quotient map by the action of $S_3$ on $\mathcal{T}_\Delta$ that acts by permutations of the labellings. Alternatively, in the  above parametrization of $\mathcal{T}_\Delta$, the group $S_3$ acts on triples $(a,b,c)$ by permuting the coordinates.  The parametrization of $\mathcal{M}_\Delta $ is then immediate from the fact that we can always choose a permutation so that the entries satisfy $a\leq b\leq c$. 
\end{proof}



    


\textit{Remark.} The above parametrizations also describe these spaces as planar regions; in particular, since $c=2-a-b$, it can be written in terms of the $a,b$ parameters, one can identify $\mathcal{T}_\Delta$ with its projection to the $ab$-plane -- see Figure~\ref{fig:ModuliSpaceofTriangles}. 
Thus,
\begin{equation}\label{ab-param}\mathcal{T}_\Delta \cong \{ (a,b) \in \mathbb{R}_+^3\ \vert\ \ a, b<1 \text{ and } a+b>1\}
\end{equation}
since $c < 1$ if and only if $a+b >1$. 
We shall use this description in the final section. For completeness, the projection also provides the identification $\mathcal{M}_\Delta \cong \{(a,b) \in \mathbb{R}^3_+\ \vert\ a\leq b, a+ 2b\leq 2 \text{ and } a+b>1\} $ which is also derived in \cite{Gaspar}.

\begin{figure}
\centering
\includegraphics{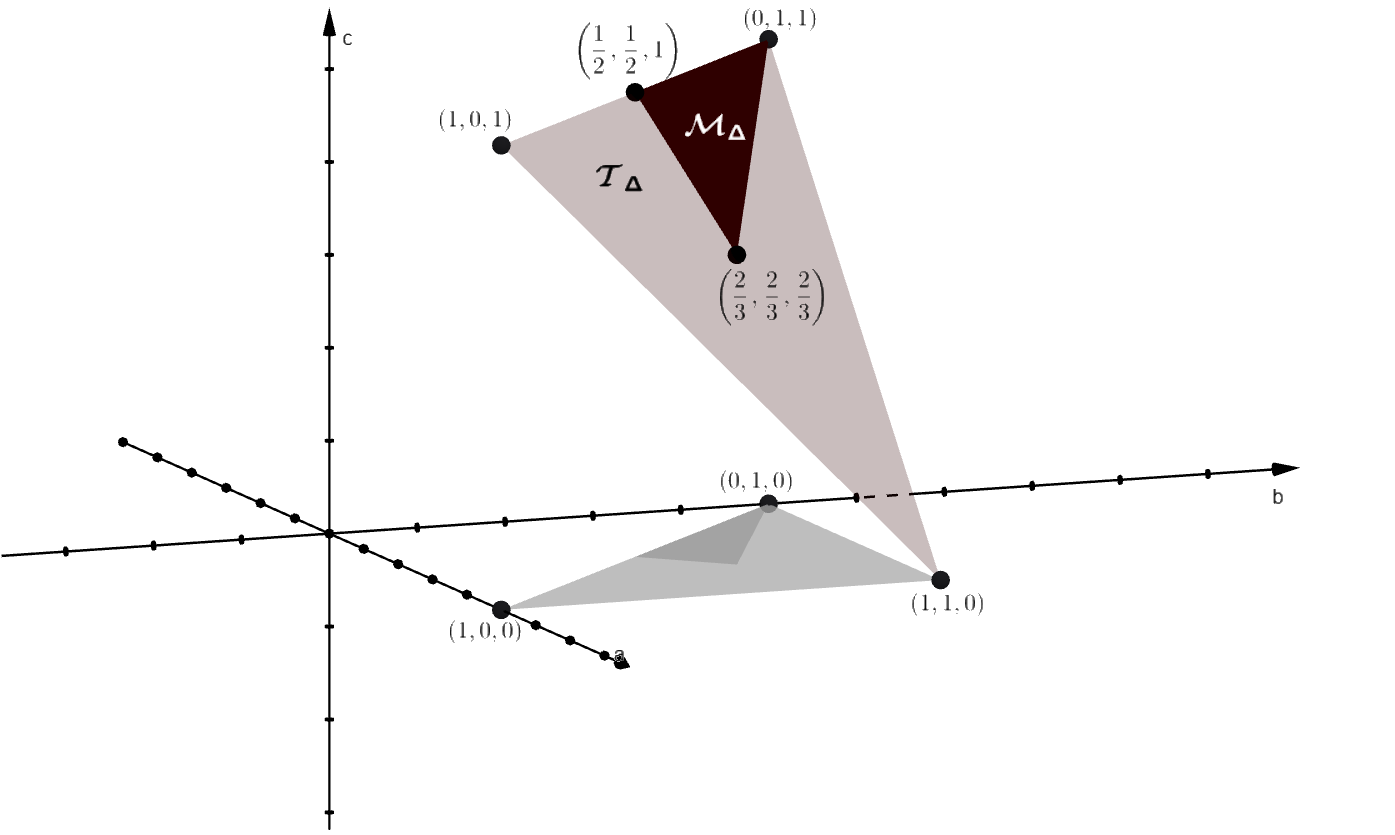}
\caption{The Teichm\"{u}ller space of triangles $\mathcal{T}_\Delta$ as parameterized in Lemma \ref{param}, together with its projection to the $ab$-plane. The subset corresponding to $\mathcal{M}_\Delta$ is also shown, shaded darker.}
\label{fig:ModuliSpaceofTriangles}
\end{figure}

\medskip 

In this article, we are interested in understanding the subset of this moduli space that corresponds to lattice triangles. We now precisely describe how a  lattice triangle determines a point in $\mathcal{M}_\Delta$; note that in the following definition, normalizing with the \textit{semi-perimeter} ensures that $a+b+c =2$, so the triple $(a,b,c)$ indeed lies in $\mathcal{M}_\Delta$ as parametrized in Corollary \ref{cor:param}.

\begin{defn}[Lattice triangle $\to$ Points in $\mathcal{M}_\Delta$]\label{corr} A lattice triangle ${\Delta}$ determined by its set of vertices $A,B,C \in \mathbb{Z}^2$ specifies a unique point in $\mathcal{M}_\Delta$, namely the one given by $(a,b,c)$ where
\begin{equation}\label{eq:norm}
a=\frac{1}{s}\lVert B-A\rVert,\ b=\frac{1}{s}\lVert C-B\rVert,\ c=\frac{1}{s}\lVert A-C\rVert
\end{equation}
where $s = \frac{1}{2}(\lVert B-A\rVert + \lVert C-B\rVert + \lVert A-C\rVert)$ is the ``semi-perimeter", and we choose an ordering such that $a\leq b\leq c$. 
\end{defn}

\textit{Remark.} Corresponding to a lattice triangle $\Delta$, we usually obtain a set of $6$ points in $\mathcal{T}_\Delta$ corresponding to all the possible ways of assigning the three labels to the sides of $\Delta$. The exceptions are lattice triangles which are isosceles, in which case, because of the additional symmetry, it will determine $3$ such points in $\mathcal{T}_\Delta$ (see Figure~\ref{fig:IsoscelesAndEquilateral}). We use this symmetry to plot points in $\mathcal{T}_\Delta$ corresponding to lattice triangles, in Figure \ref{fig:plot}.

\begin{figure}
\centering
\includegraphics{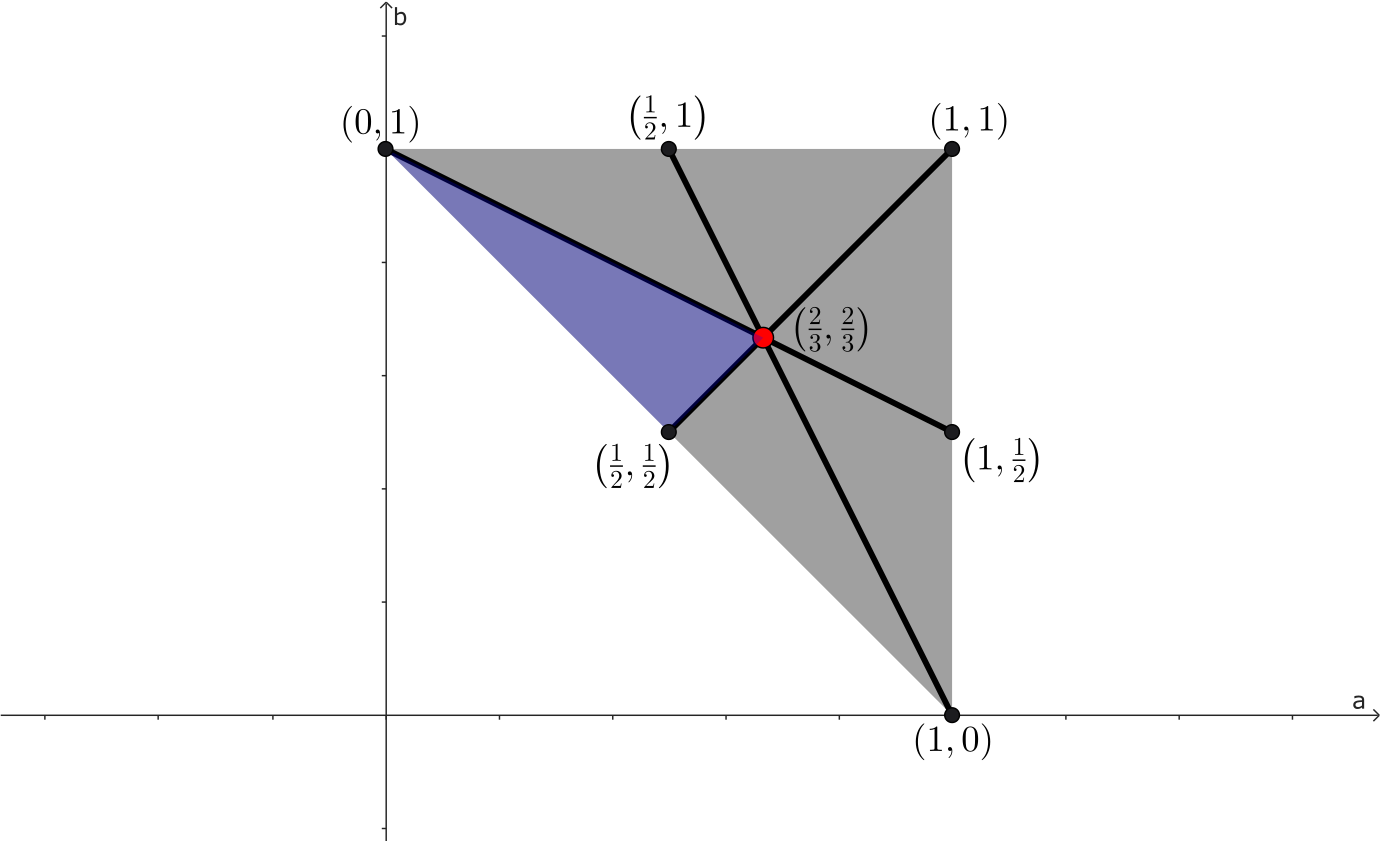}
\caption{The space $\mathcal{T}_\Delta$ is parametrized by two normalized side-lengths $a,b$, since the third side-length $c=2-a-b$. The loci of isosceles triangles (where two side-lengths are equal) are the three line segments (shown in bold); they meet at the equilateral triangle (shown in red). The fundamental domain of the action of the permutation group $S_3$ (shown shaded in blue), is the moduli space $\mathcal{M}_\Delta$.}
\label{fig:IsoscelesAndEquilateral}
\end{figure}

\subsection{Simultaneous Dirichlet approximation}
\label{subsection:SimultaneousDirichletApproximation}
Our main tool for approximating a similarity class of triangles with lattice triangles will be a variant of Dirichlet’s Approximation Theorem, which is about approximating an (irrational) real number by a rational.

\begin{thm}[Dirichlet]\label{dirich}
Given \( x \in \mathbb{R} \) and any $\varepsilon > 0$, there exist integers \( M , N\), such that $\left| Mx - N \right| < \varepsilon$.
\end{thm}

We omit the proof here, as we shall use the same idea (namely, the pigeonhole principle) in our proof of a slightly extended version below. We remark here that there are deeper results in number theory concerning how well real numbers are approximable by rationals (see, for example, \cite{DS}). 

\medskip


\begin{lem}[Simultaneous Dirichlet]\label{lem:SimultaneousDirichlet}
    Given real numbers \(x, y \in \mathbb{R}\) and any \(\varepsilon > 0\), there exists integers \(M, N_x, N_y\) such that
\(
\vert M x - N_x \vert < \varepsilon \text{ and } \vert M y - N_y \vert < \varepsilon
\).
\end{lem}

\begin{proof}
Recall that for a real number $x$, $\lfloor x\rfloor$ denotes the largest integer \textit{not greater than} $x$, and $\{x\}:= x - \lfloor x \rfloor$ denotes the fractional part. 
Choose an integer $N$ such that  $\frac{1}{N} < \varepsilon$. 

We divide the unit square \([0,1]^2\) into \(N^2\) equal boxes, each of size \( \frac{1}{N} \times \frac{1}{N} \), so in particular the height and width of each box is less than \( \varepsilon \). Consider the points $P_0, P_1, \ldots, P_{N^2}$ defined as
\[
P_k := \left( \{kx\}, \{ky\} \right)\in [0,1]^2
\]
for $0\leq k \leq N^2$. 
These are $N^2+1$ points that lie in $N^2$ boxes. Therefore, by the pigeonhole principle, there exist indices, say \( 0 \leq i < j \leq N^2 \) such that $P_i, P_j$ lie in the same box $B$ of width and height $\frac{1}{N}$. As $P_i = \left(\{ix\}, \{iy\}\right)$ and $P_j = \left(\{jx\}, \{jy\}\right)$, we obtain 
\[\vert\{jx\} - \{ix\} \vert \leq \frac{1}{N} < \varepsilon \text{ and } \vert \{jy\} - \{iy\} \vert \leq \frac{1}{N} < \varepsilon. \]

Let  $M := j - i > 0$ and note that 
\begin{align*}
\vert M x - (\lfloor j x\rfloor - \lfloor i x\rfloor) \vert &= \vert (j - i) x - (\lfloor j x\rfloor - \lfloor i x\rfloor) \vert \\
                                                            &= \vert j x - \lfloor j x \rfloor - \left( i x - \lfloor i x \rfloor\right) \vert \\
                                                            &= \vert \{j x\} - \{i x\} \vert \\
                                                            &< \varepsilon. 
\end{align*}
Similarly, $\vert M y - (\lfloor j y\rfloor - \lfloor i y\rfloor) \vert < \varepsilon$. We take $N_x \coloneqq \lfloor j x\rfloor - \lfloor i x\rfloor$ and $N_y \coloneqq \lfloor j y\rfloor - \lfloor i y\rfloor$ to finish the proof. \end{proof}

\section{Proof of the Density result} 

\subsection{Approximating equilateral triangles}

We start with proving a special case of Theorem \ref{thm:main}, which uses the classical Dirichlet Approximation Theorem. That is, we shall prove that although a lattice triangle cannot be equilateral (see Lemma \ref{equi}), we have lattice triangles that \textit{approximate} an equilateral triangle arbitrarily well. In other words, points in $\mathcal{M}_\Delta$ corresponding to lattice triangles (see  Definition \ref{corr})  come arbitrarily close to $c_0 = \left(\frac{2}{3}, \frac{2}{3}, \frac{2}{3}\right)$. Here, we are using the parametrization in Corollary \ref{cor:param}, where $c_0$ represents the similarity class of an equilateral triangle (see Figure \ref{fig:IsoscelesAndEquilateral}). In what follows, the distance between two triples $(a,b,c)$ and $(a^\prime, b^\prime, c^\prime)$ is the usual Euclidean norm $\lVert (a,b,c) - (a^\prime, b^\prime, c^\prime)\lVert \ := \sqrt{(a-a^\prime)^2 + (b-b^\prime)^2 + (c-c^\prime)^2}$.

\begin{prop}
    There is sequence of lattice triangles $\{\Delta_n\}_{n\geq 1}$ such that the corresponding set of points in $\mathcal{M}_\Delta$ has  $c_0= \left(\frac{2}{3}, \frac{2}{3}, \frac{2}{3}\right)$ as an accumulation point. 
\end{prop}

\begin{proof}

For each integer $n \geq 1$, we construct an equilateral triangle $E_n$ corresponding to the point $c_0 \in \mathcal{M}_{\Delta}$. Define $E_n$ to be the equilateral triangle in $\mathbb{R}^2$ with vertices $(0,0), (2n, 0)$ and $(n, n \sqrt{3})$. Every side of $E_n$ has length $2n$ and dividing by the semi-perimeter we see that side lengths of $E_n$ give us the tuple $c_0$. Note that $E_n$ is not a lattice triangle for any $n$. 

By Theorem~\ref{dirich}, we know that for any $\varepsilon > 0$, there exists integers $M_\varepsilon, N_\varepsilon$ such that \[\vert M_\varepsilon \sqrt{3} -  N_\varepsilon \vert < \epsilon.\]

Let the lattice triangle given by the vertices $(0,0), (2M_\varepsilon, 0), (M_\varepsilon, N_\varepsilon)$ be denoted by $\Delta_\varepsilon$. 
Let $\epsilon_k$ be a sequence such that $\epsilon_k \rightarrow 0$ as $k\to \infty$. Then, the distance between $\Delta_{\epsilon_k}$ and $E_{M_{\epsilon_k}}$ equals $\lvert M_{\epsilon_k} \sqrt{3} - N_{\epsilon_k} \rvert \to 0 $ as $k\to \infty$. This implies that the  sequence of lattice triangles $\Delta_{\epsilon_k}$ comes arbitrarily close to $c_0$ and thus has $c_0$ as an accumulation point.
\end{proof}

\subsection{Approximating general triangles}

Showing the subset of $\mathcal{M}_\Delta$ corresponding to lattice triangles is dense is equivalent to proving that for any tuple $\left (a, b, c\right) \in \mathcal{M}_{\Delta}$ and for any $\varepsilon > 0$, there exists a lattice triangle $\Delta_\varepsilon$ whose corresponding point in $\mathcal{M}_\Delta$, given by  normalised edge-lengths,  is $\varepsilon$-close to $\left(a, b, c\right)$. Recall that a pair of tuples of real numbers are said to be $\epsilon$-close if  the Euclidean distance between them is less than $\varepsilon$. 

Let $T = (a, b, c) \in \mathcal{M}_\Delta$. If $T$ can be represented by a lattice triangle, we are done. Suppose not; we can assume, by applying a similarity if need be, that two of the vertices of $T$ are $A = (0,0)$ and $B =(1,0)$. Let the third vertex be $C = (x,y)$. 

Note that in contrast to the case of an equilateral triangle, in the general case \textit{both} $x$ and $y$ may be irrational numbers. We shall need here the ``simultaneous Dirichlet approximation" result that we proved in \S \ref{subsection:SimultaneousDirichletApproximation}.  That is, by Lemma~\ref{lem:SimultaneousDirichlet}, for any $\varepsilon > 0$, there exists integers $M, N_x, N_y$ such that 
\[ \vert M x - N_x \vert < \frac{\varepsilon}{\sqrt{2}} \text{ and } \vert M y - N_y \vert < \frac{\varepsilon}{\sqrt{2}}.\]

Now, consider the point $(Mx, My)$, then we see that
\[\lVert \left( Mx, My\right) - \left( N_x, N_y \right) \rVert  = \sqrt{(Mx - N_x)^2 + (My - N_y)^2}< \varepsilon\]

In other words, given an $\varepsilon > 0$, there exists a positive integer $M$ such that $(Mx, My)$ is $\varepsilon$-close to the lattice point $(N_x, N_y)$. Since the map \eqref{eq:norm} from vertices of triangle  to its normalized triple of edge-lengths $(a,b,c)$ is continuous, a small perturbation of a vertex produces a small change in the corresponding point in $\mathcal{M}_\Delta$. In particular, the point $T$ in $\mathcal{M}_\Delta$ that corresponds to the triangle with vertices $(0,0), (M, 0)$ and $(Mx, My)$ is $\varepsilon$-close to that corresponding to the lattice triangle formed by vertices $(0,0), (M,0)$ and $(N_x, N_y)$, and we are done. \qed

\section{Distribution in moduli space}

\subsection{The question of equidistribution}
\label{subsection:TheQuestionOfEquidistribution}
The Teichm\"{u}ller space of triangles $\mathcal{T}_\Delta$ parametrized in Lemma \ref{param} is equipped with a natural measure: the area measure in the $ab$-plane. Here, recall that $\mathcal{T}_\Delta$ is identified with the two-dimensional region given by \eqref{ab-param}, using the fact that the third side $c$ equals $2-a-b$. Then  for a subset  $S \subset \mathcal{T}_\Delta$, we define its measure $\mu(S) := \text{Area}(S)$. 

This equips $\mathcal{M}_\Delta$ with a measure, that we also denote by $\mu$, since the action of $S_3$ on $\mathcal{T}_\Delta$ (see Corollary \ref{cor:param}) is measure-preserving. The fact that it is measure-preserving can be seen as follows:  in the parametrization using normalized lengths $a,b,c$ given by Lemma \ref{param}, the action of $S_3$ permuting the labels of the sides simply permutes the coordinates  $(a,b,c)$ on the plane $a+b+c=2$ and preserves the Euclidean area on that plane. Moreover, the projection of any subset of $a+b+c=2$ to the $ab$-plane merely scales its measure by a scaling factor $1/\sqrt{3}$, and hence the action also preserves $\mu$.

\medskip 

In what follows, we shall need the \textit{total} measure of moduli space, that we now record. 
\begin{lem}\label{tot}
    The total measure $\mu(\mathcal{T}_\Delta)=\frac{1}{2}$, while the total measure $\mu(\mathcal{M}_\Delta)=\frac{1}{12}$.
\end{lem}

\begin{proof}
    The first measure is the area of the triangle with vertices $(1,0), (0,1)$ and $(1,1)$ which is the projection of $\mathcal{T}_\Delta$ to the $ab$-plane (\textit{c.f.} Figure~\ref{fig:IsoscelesAndEquilateral}). The second measure is $1/6$-th of the former, since $\lvert S_3 \rvert = 6$ and $\mathcal{M}_\Delta = \mathcal{T}_\Delta/S_3$, and as explained above, the action of $S_3$ preserves $\mu$.
\end{proof}

In this section, we will try to examine the question of whether the points in $\mathcal{M}_\Delta$ corresponding to lattice triangles are distributed evenly with respect to this measure. To formulate this question more precisely, we first define the following notions of a ``weighted set" in a measure-space $X$, and its corresponding  ``Dirac measure". We need this as several different lattice triangles may have the same similarity type, so we would need to count \textit{with multiplicity}. For simplicity, we assume that the space $X$ is a planar subset, equipped with the area measure $\mu$, as it is in our setting.


\begin{defn}[Weighted set of points]\label{defn:wted} Given a space $X$, a (finite) weighted set of points $S$ is a finite subset $\{p_1,p_2,\ldots,p_n\} \subset X$, together with a corresponding set of  positive integers (the ``weights") $w_1,w_2,\ldots w_n$.  In particular, if  $\widehat{S}$ is a finite sequence of points in $X$, which are possibly not all distinct, then it defines a  weighted set $S$ 
by associating to each point of the (pairwise distinct) points, a weight equal to the multiplicity, i.e., number of times it appears in the sequence.  Moreover, we define the \emph{total weight} of $S$  to be $w(S) := \sum\limits_{i=1}^n w_i$, i.e., the sum of weights. 
\end{defn}

\begin{defn}[Dirac measure]\label{defn:dirac} For a space $X$, and a point $p\in X$ the Dirac measure  $\delta_p$ is defined by 
\[
\delta_p(A) = \begin{cases}
  0, & p \notin A, \\
  1, & p \in A
\end{cases}
\]
for any subset $A$. We extend this to a Dirac measure on a weighted set $S$, by defining
\begin{equation}
    \delta_S: = \sum\limits_{i=1}^n w_i \delta_{p_i}
\end{equation}where $p_1,p_2,\ldots,p_n$ are the points of $S$ with corresponding (positive integer) weights $w_1,w_2,\ldots, w_n$. 
\end{defn}

We may now define a notion of equidistribution of weighted sets in a planar region (also referred to as \textit{uniform} distribution of sequences, see \cite{Uniform-book}).

\begin{defn}[Equidistribution]\label{defn:equi} Let $X\subset \mathbb{R}^2 $ be a planar subset with area measure $\mu$, and let $\{S_n\}$ be a sequence of finite weighted subsets of $X$, where the total weights $w(S_n) \to \infty$. We say that this sequence \textit{equidistributes} in $X$ if the corresponding Dirac measures converge to $\mu$ after normalizing, that is,
\begin{equation}\label{eq:equi}
    \frac{1}{w(S_n)}\delta_{S_n}(R) \to \frac{\mu(R)}{\mu(X)}
\end{equation}
as $n\to \infty$, for any domain $R \subset X$.
\end{defn}
\medskip


\medskip

\textit{Examples.} (i) In one lower dimension, when $X = [0,1]$, an example of when equidistribution holds connects back to Dirichlet's theorem,  namely if $S_n = \{ \{j\sqrt 3\}\ \vert\ 1\leq j\leq n\} $, then it is a basic fact of ergodic theory due to Weyl that $S_n$ equidistributes in $[0,1]$. (See \cite{Weyl-wiki}.) 

(ii) If $X = [0,1] \times [0,1]$, and $S_N = \left\{\left(\frac{i}{N}, \frac{j}{N}\right)\ \vert\ 1\leq i,j\leq N\right\}$ for each $N\geq 1$, then it is easy to prove that the sequence $\{S_N\}$ equidistributes in $X$. 

\medskip

In our setting, we let $S_N$ be the weighted set of points of $\mathcal{M}_\Delta$ corresponding to lattice triangles with vertices inside the square $[-N,N] \times [-N,N]$, where the weight equals the ``multiplicity",  as in Definition \ref{defn:wted}. In particular, we consider the underlying set of (pairwise distinct) lattice triangles in $[-N,N]^2$, and assign a weight to each that is the number of  lattice triangles in the square that are in the same similarity class. We shall answer the following question.

\begin{ques}\label{ques:equi} 
    Do the weighted subsets $S_N$ described above equidistribute in $\mathcal{M}_\Delta$?
\end{ques}

 Our Theorem \ref{thm:nequi}, that we prove in the next section, says that this has a negative answer. It could be interesting to consider other ways of exhausting the set of lattice triangles (by using disks of increasing area, instead of squares), and we leave that to the interested reader. For other questions that are unanswered, see \S \ref{section:epilogue}.

\subsection{Proof of Theorem \ref{thm:nequi}}

Recall that a triangle with side-lengths $(a,b,c) \in \mathbb{R}_+$ is 
\begin{itemize}
    \item \textit{right-angled} if the tuple  satisfies $a^2 +b^2=c^2$, and
    \item \textit{obtuse} if $c^2>a^2 +b^2$,
\end{itemize}
where the right-angle in the former case, and the obtuse angle in the latter, is opposite the side of length $c$.

\medskip
We shall need the following result from \cite{Langford} (see also the expanded version in \cite{Langford2}). 

\begin{prop}\label{prop:langf} 
    If the three vertices of a triangle $\Delta$ are chosen randomly from a unit square with respect to the uniform measure, then the probability that $\Delta$ is obtuse equals $\frac{97}{150} + \frac{\pi}{40} \approx 0.72520648$.
\end{prop}

The fact that it is more likely for random triangle to be obtuse  (for various interpretations of the word ``random") has had a long and interesting history -- see, for example, \cite{Portnoy}.

\medskip

Note that in the parametrization of $\mathcal{T}_\Delta$ given by Lemma \ref{param}, the subset corresponding to obtuse triangles with obtuse angle opposite to $c$ is the  region $$O_c = \{ (a,b,c)\in \mathbb{R}_+^3\ \vert\ a,b,c<1, a+b+c=2, c^2 >a^2+b^2\}.$$ There are similar regions $O_a$, $O_b$ that are symmetric with respect to a cyclic permutation of labels (see Figure~\ref{fig:ObtuseTriangles}).

The subset of similarity classes of obtuse triangles with edge-labellings is $O_a \cup O_b \cup O_c$, which is invariant under the action of $S_3$ (that acts by permuting coordinates), and determines the subset $\mathcal{O}$ of similarity classes of obtuse triangles in $\mathcal{M}_\Delta$. Since this action is also measure-preserving, the measure $\mu(\mathcal{O}) = \frac{1}{6} \cdot \mu(O_a \cup O_b\cup O_c) $.

\medskip 



\begin{figure}
\centering
\includegraphics{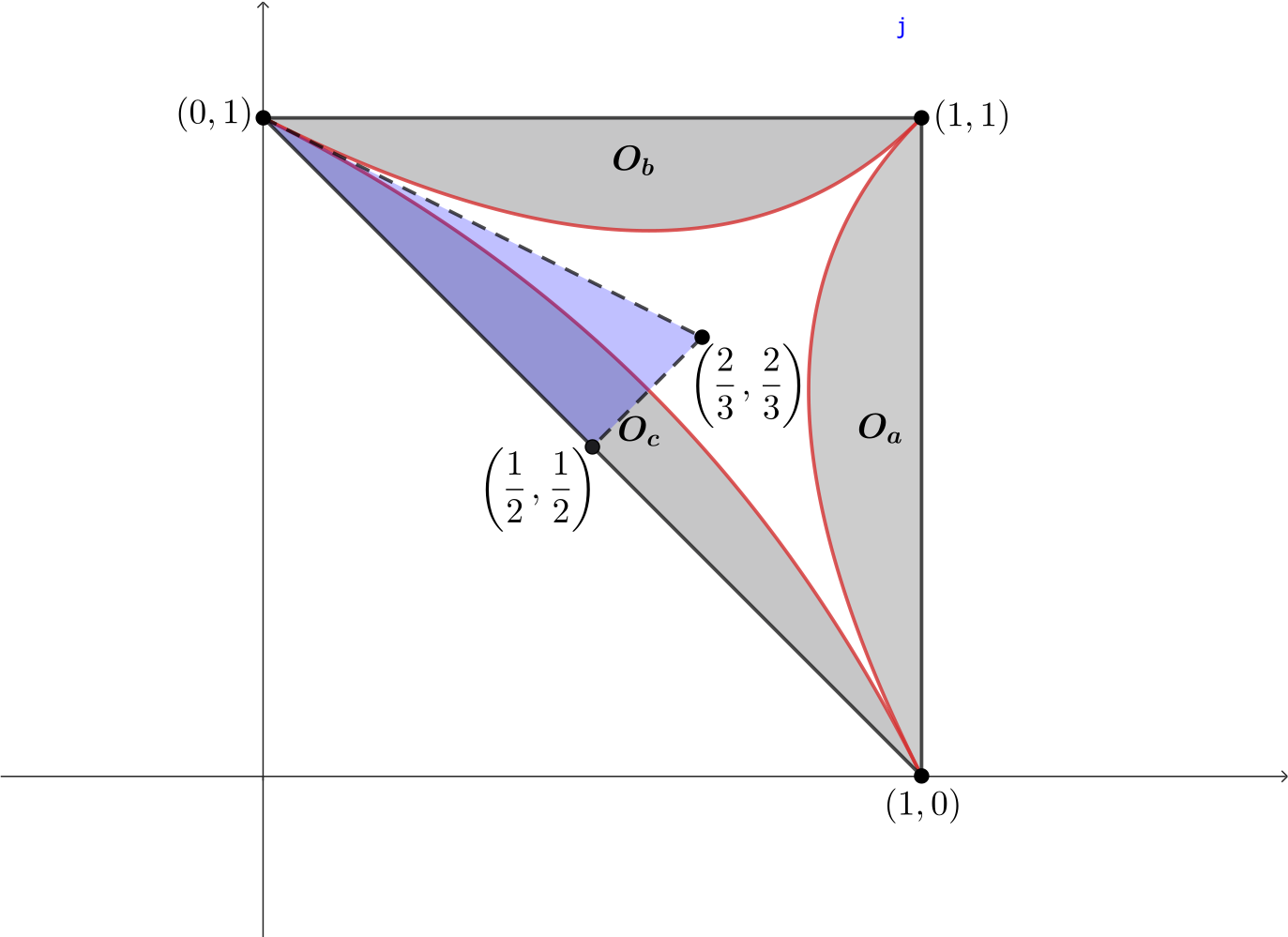}
\caption{The subset of $\mathcal{T}_\Delta$ corresponding to obtuse triangles defines the three regions $O_a,O_b$ and $O_c$ (shaded grey). The quotient of $O_a \cup O_a \cup O_b$ by the action of $S_3$ determines the subset $\mathcal{O}$ of similarity classes of obtuse triangles in $\mathcal{M}_\Delta$ (shaded purple). }
\label{fig:ObtuseTriangles}
\end{figure}

\begin{lem}\label{lem:muobt}
    The measure $\mu(O_a \cup O_b\cup O_c)= \tfrac{9}{2}-6\ln 2 \approx 0.341117$.
\end{lem}

\begin{proof}
    By cyclic symmetry, $\mu(O_a \cup O_b\cup O_c) = 3\mu(O_c)$ so it suffices to compute the latter.

    Recall that the angle opposite $c$ is a right angle exactly when $c^2 = a^2 +b^2$, which using $c=2-a-b$ yields 
\[
(2-a-b)^2 = a^2+b^2
\quad\Longleftrightarrow\quad ab-2a-2b+2=0.
\]
Solving for $a$ then defines the curve
\[
a=\frac{2(1-b)}{2-b}, \qquad b\in(0,1)
\]
in the $ab$-plane, which  is the projection of the locus of right-angled triangles in $\mathcal{M}_\Delta$. 

The subset $O_c$ corresponding to the set of obtuse triangles thus projects to the region $R$ in the $ab$-plane bounded by the curve above, and the line segment  $a+b=1$. (See Figure \ref{fig:ObtuseTriangles}.)  This description of the locus was also observed in \cite{Stewart} and \cite{Gaspar}. We then have $$\mu(O_c) = \text{Area}(R) =\int_0^1 \left(\frac{2(1-b)}{2-b} - (1-b)\right)\,db. $$

Simplifying the integrand, we obtain
\[
\frac{2(1-b)}{2-b} - (1-b) 
= (1-b)\cdot\frac{b}{2-b}
= b+1+\frac{2}{b-2}.
\]

Hence the integral evaluates to yield
\[
\text{Area}(R)
=\Big[\tfrac{b^2}{2}+b+2\ln(2-b)\Big]_{0}^{1}
=\tfrac{3}{2}-2\ln 2.
\]

The total  measure of obtuse triangles (with labelling)  is $3\mu(O_c)
=\tfrac{9}{2}-6\ln 2 \approx 0.341117$.  \end{proof}

\medskip

Now, let $\mathcal{L}_N$ denote the set of lattice triangles in $[-N,N]^2$, and let $S_N$ be the corresponding weighted subset in $\mathcal{M}_\Delta$. If we assume that $S_N$ equidistributes in $\mathcal{M}_\Delta$ as $N\to \infty$, then by setting $R = \mathcal{O} = \big\{\text{ obtuse lattice triangles up to similarity}\big\}$ in \eqref{eq:equi}, we obtain
\begin{equation}\label{eq2}
    \frac{1}{w(S_N)}\delta_{S_N}(\mathcal{O})  \to  \frac{\mu(\mathcal{O})}{\mu(\mathcal{M}_\Delta)} = \frac{\mu(O_a\cup O_b\cup O_c)}{\mu(\mathcal{T}_\Delta)} = \frac{\tfrac{9}{2}-6\ln 2}{\tfrac12}
 \approx 0.682234
\end{equation}
as $N\to \infty$. Here, the first equality holds because both the numerator and denominator get scaled by a factor of $6$, and the second equality follows from Lemmas \ref{tot} and \ref{lem:muobt}.

However, by Definition \ref{defn:dirac} we have that the left hand side equals
\begin{equation}\label{eq2b}
 \frac{1}{w(S_N)}\delta_{S_N}(\mathcal{O}) = \frac{\lvert \{\text{triangles in } \mathcal{L}_N\text{ that are obtuse} \}\rvert}{\lvert \mathcal{L}_N\rvert} 
\end{equation} 
which, after rescaling the square $[-N,N] \times [-N,N]$ by a factor of $1/2N$,  converges as $N\to \infty$ to the probability that a  random triangle in the unit square is obtuse (\textit{c.f.} Example (ii) following Definition \ref{defn:equi}). By Proposition \ref{prop:langf}, the right-hand side of \eqref{eq2b} is approximately 0.72520648. 

This is a contradiction -- on one hand, equidistribution would force the limiting proportion of obtuse triangles to be $\approx 0.682234$ by our area computation, and on the other hand, Langford's result in geometric probability shows it must be $\approx 0.72520648$. Since these values differ, equidistribution cannot hold and the subsets $\{S_N\}_{N\geq 1}$ do not equidistribute in $\mathcal{M}_\Delta$. This completes the proof of Theorem \ref{thm:nequi}.

\section{Epilogue}\label{section:epilogue}

As mentioned in the Introduction, we do not know the expression for the exact limiting probability distribution in $\mathcal{M}_\Delta$ for the weighted subsets $S_N$ that we considered, as $N\to \infty$.

\begin{figure}
\centering
\includegraphics{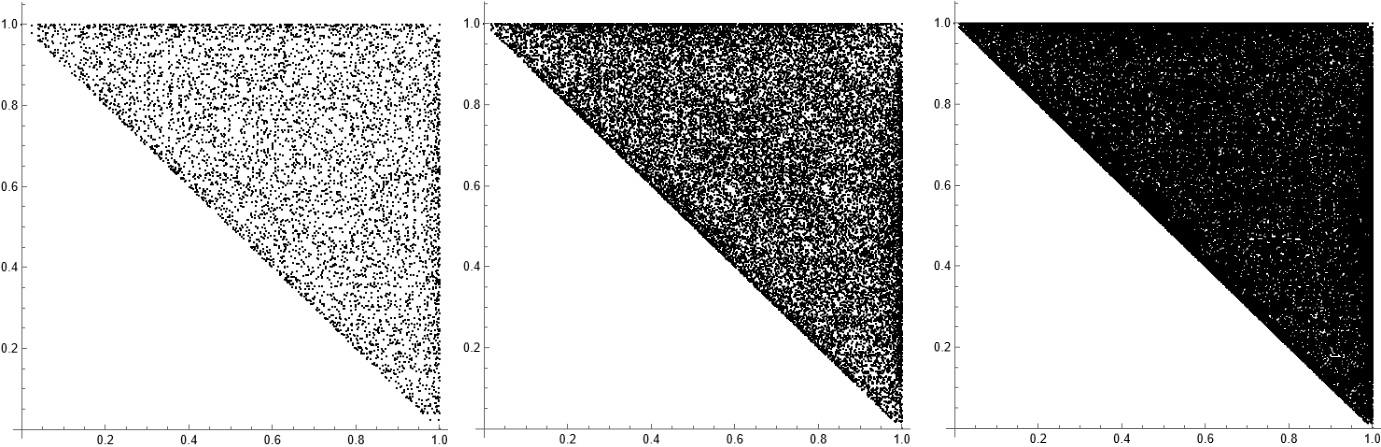}
\caption{Points in $\mathcal{T}_\Delta$ corresponding to a random sample of $10^3$ points (left), $5\times 10^3$ points (middle), and $1.5 \times 10^4$ points (right) in $\mathcal{S}_{100}$,\ i.e.,  lattice triangles in $[-100, 100] \times [-100,100]$. In the middle figure notice that more points seem to concentrate close to the sides; this is expected as by Langford's result, $\approx 72\%$ of obtuse triangles needs to fit into $\approx 68\%$ of the area of $\mathcal{T}_\Delta$.  The distribution away from the sides is visually close to being uniform, which is confirmed by density plots.}
\label{fig:plot}
\end{figure}

By  rescaling of the square subset $\mathbb{Z}^2 \cap ([-N,N] \times [-N, N])$ by a factor of $1/2N$, as before,  we can reformulate this question as:

\begin{ques}\label{ques:dist}
    Given a triple of points $A,B,C$ in the unit square $[0,1]\times [0,1]$ chosen independently and randomly with respect to the uniform distribution, what is the probability distribution of the normalized triple  of side-lengths $(a,b,c) \in \mathbb{R}^3_+$ as defined in \eqref{eq:norm}?
\end{ques}

\medskip 


There is a simpler related question in geometric probability regarding \textit{pairs} of random points (rather than triples)  whose answer is known.

\begin{ques} What is the expected distance between two points in the unit square $[0,1] \times [0,1]$ chosen randomly with respect to the uniform distribution?
\end{ques}

The answer to this is approximately 0.5214 (see, for example, the solution to Problem 95-6 in \cite{Boehm}), and the probability distribution of this random distance has been studied for a long time (see, for example, \cite{Ghosh}).  This distribution is far from being uniform -- for a plot see \cite{WolframSite}.

\medskip 

We end with some other questions and pointers to the literature that an interested reader may pursue.

\subsection{Forgetting weights}

For Theorem \ref{thm:nequi} we considered the sequence of \textit{weighted} sets  $S_N$ in $\mathcal{M}_\Delta$ corresponding to lattice triangles in $[-N,N]^2$; a natural question that arises is what distribution do we get if one considers only the underlying sets (forgetting the weights). 

\begin{ques} Let $S_N = \sum\limits_{i=1}^M w_i p_i$ where $w_i$ is the weight/multiplicity of the point $p_i$. Then does the sequence of underlying sets  points $S_N^\prime = \{p_1,p_2,\ldots, p_M\}$ equidistribute in $\mathcal{M}_\Delta$, that is, does $\displaystyle\frac{1}{M} \sum\limits_{i=1}^M \delta_{p_i} \to \mu$ as $N \to \infty$?
\end{ques}

(It is not hard to see that the cardinality of the sets $M\to \infty$ as $N \to \infty$.) 

\medskip

Numerical experiments suggest that the answer to the above question is again negative, and that the limiting distribution differs from that of the sequence of weighted sets (see Figure~\ref{fig:probability_plot}).

\begin{figure}[h]
\centering
\includegraphics{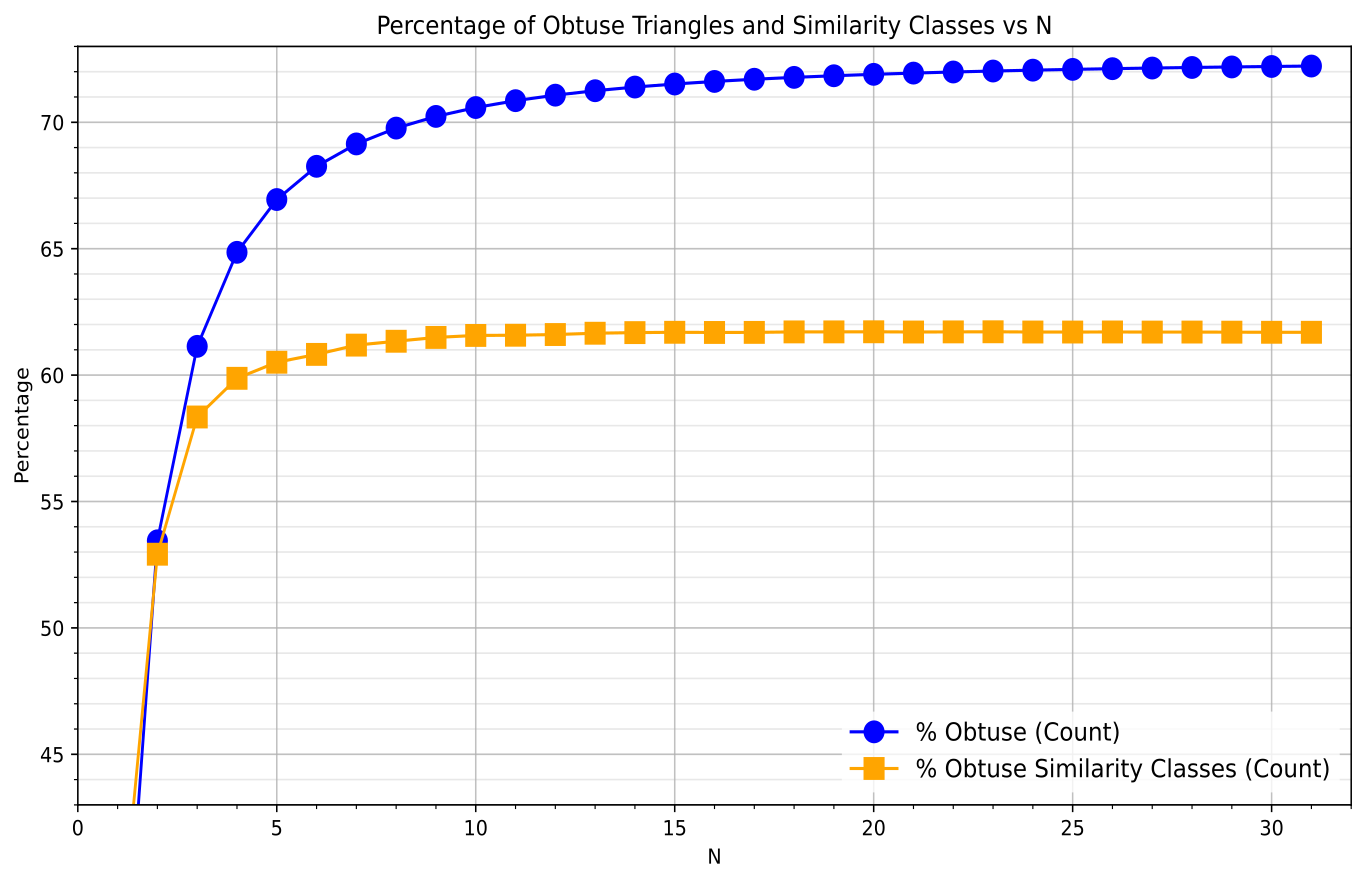}
\caption{The percentage of obtuse triangles amongst all triangles in $[-N,N]^2$, when $2\leq N\leq 31$ is plotted in blue, and the percentage of \textit{similarity classes} of obtuse triangles amongst all similarity classes is plotted in orange. Note that the blue curve approaches $\approx 72\%$ as expected by Langford's result, while the orange curve approaches $\approx 62\%$. }
\label{fig:probability_plot}
\end{figure}

\subsection{Rate of convergence}


It would be interesting to obtain a quantitative understanding of how fast the weighted subsets $S_N$ become dense in $\mathcal{M}_\Delta$ as $N\to \infty$, that is,\ obtain a rate of convergence to the limiting distribution, as that may have connections with number-theory. Indeed, the problem of counting integer lattice points in a disk of radius $R$, or equivalently pairs of integers the sum of whose squares is at most $R^2$, is the famous Gauss Circle problem (see, for example \cite{Gausscircle-wiki}).

\subsection{Higher dimensions}

It is easily seen that the $3$-dimensional integer lattice $\mathbb{Z}^3$ has lattice tetrahedra that are \textit{regular} (the analogue of an equilateral triangle). In general, a regular $n$-simplex embeds in the $n$-dimensional lattice when $n=3,7,8,9,11,15,17, 19,23,\ldots$  (see \cite{oeis} and \cite{Schoenberg} for the exact result). 

The moduli space of tetrahedra in $\mathbb{R}^3$ (or  simplices in $\mathbb{R}^n$ where $n>2$) is higher-dimensional, so it is more difficult to visualize points corresponding to those with vertices in the integer lattice, like in Figure \ref{fig:plot}. However, we expect that the analogues of the density result (Theorem \ref{thm:main}) and non-equidistribution (Theorem \ref{thm:nequi}) will continue to hold in this setting.

\vspace{0.5in} 
\textbf{Interest disclosure:} There are no relevant financial or non-financial competing interests to report.

\bibliographystyle{amsalpha}
\bibliography{references}

@article{Shirali,
  author    = {Shirali, Shailesh},
  title     = {{Lattice point geometry, Part-I}},
  journal   = {Azim Premji University At Right Angles},
  year      = {2019},
  pages     = {5--10},
  issn      = {2582-1873},
  publisher = {Azim Premji University together with Community Mathematics Center, Rishi Valley}
}

@article {Beeson,
    AUTHOR = {Beeson, Michael J.},
     TITLE = {Triangles with vertices on lattice points},
   JOURNAL = {Amer. Math. Monthly},
  FJOURNAL = {American Mathematical Monthly},
    VOLUME = {99},
      YEAR = {1992},
    NUMBER = {3},
     PAGES = {243--252},
}

@misc{HamkinsBlog,
  author  = {Hamkins, Joel D.},
  title   = {There are no nondegenerate regular polygons in the integer lattice, except for squares},
  year    = {2013},
  month   = {Jan},
  day     = {14},
  note = {\url{https://jdh.hamkins.org/no-regular-polygons-in-the-integer-lattice//}},
}

@misc{WolframSite,
  author  = {Weisstein, Eric W.},
  title   = {Square Line-Picking},
  note = {From MathWorld--A Wolfram Resource, at \url{https://mathworld.wolfram.com/SquareLinePicking.html}},
}

@incollection {Behrend,
    AUTHOR = {Behrend, K.},
     TITLE = {Introduction to algebraic stacks},
 BOOKTITLE = {Moduli spaces},
    SERIES = {London Math. Soc. Lecture Note Ser.},
    VOLUME = {411},
     PAGES = {1--131},
 PUBLISHER = {Cambridge Univ. Press, Cambridge},
      YEAR = {2014},
}

@article{Langford,
    author = {Langford, Eric},
    title = {The probability that a random triangle is obtuse},
    journal = {Biometrika},
    volume = {56},
    number = {3},
    pages = {689-690},
    year = {1969},
    month = {12},
    issn = {0006-3444},
    doi = {10.1093/biomet/56.3.689},
    url = {https://doi.org/10.1093/biomet/56.3.689},
    eprint = {https://academic.oup.com/biomet/article-pdf/56/3/689/636943/56-3-689.pdf},
}

@book {Uniform-book,
    AUTHOR = {Kuipers, L. and Niederreiter, H.},
     TITLE = {Uniform distribution of sequences},
    SERIES = {Pure and Applied Mathematics},
 PUBLISHER = {Wiley-Interscience [John Wiley \& Sons], New
              York-London-Sydney},
      YEAR = {1974},
}

@article {AcuteT,
    AUTHOR = {Miyachi, Hideki and Ohshika, Ken'ichi and Papadopoulos,
              Athanase},
     TITLE = {On the {T}eichm\"uller space of acute triangles},
   JOURNAL = {Monatsh. Math.},
  FJOURNAL = {Monatshefte f\"ur Mathematik},
    VOLUME = {205},
      YEAR = {2024},
    NUMBER = {3},
     PAGES = {649--666},
}

@incollection {DS,
    AUTHOR = {Davenport, H. and Schmidt, Wolfgang M.},
     TITLE = {Dirichlet's theorem on diophantine approximation},
 BOOKTITLE = {Symposia {M}athematica, {V}ol. {IV} ({INDAM}, {R}ome,
              1968/69)},
     PAGES = {113--132},

}

@article {Langford2,
    AUTHOR = {Langford, Eric},
     TITLE = {A problem in geometrical probability},
   JOURNAL = {Math. Mag.},
  FJOURNAL = {Mathematics Magazine},
    VOLUME = {43},
      YEAR = {1970},
     PAGES = {237--244},
}

@article {Schoenberg,
    AUTHOR = {Schoenberg, I.J.},
     TITLE = {Regular simplices and quadratic forms},
   JOURNAL = {J. Lond. Math. Soc.},
    VOLUME = {12},
      YEAR = {1937},
     PAGES = {48--55},
}

@misc{oeis,
    Author = {{OEIS Foundation Inc.}},
    Note = {Published electronically at \url{http://oeis.org}},
    Title = {Entry {A}096315 in The {O}n-{L}ine {E}ncyclopedia of {I}nteger {S}equences},
    Year = {2025}
}

@article {Portnoy,
    AUTHOR = {Portnoy, Stephen},
     TITLE = {A {L}ewis {C}arroll pillow problem: probability of an obtuse
              triangle},
   JOURNAL = {Statist. Sci.},
  FJOURNAL = {Statistical Science. A Review Journal of the Institute of
              Mathematical Statistics},
    VOLUME = {9},
      YEAR = {1994},
    NUMBER = {2},
     PAGES = {279--284},
}

@article {Stewart,
    AUTHOR = {Stewart, Ian},
     TITLE = {Why do all triangles form a triangle?},
   JOURNAL = {Amer. Math. Monthly},
  FJOURNAL = {American Mathematical Monthly},
    VOLUME = {124},
      YEAR = {2017},
    NUMBER = {1},
     PAGES = {70--73},

}

@article {ES,
    AUTHOR = {Edelman, Alan and Strang, Gilbert},
     TITLE = {Random triangle theory with geometry and applications},
   JOURNAL = {Found. Comput. Math.},
  FJOURNAL = {Foundations of Computational Mathematics. The Journal of the
              Society for the Foundations of Computational Mathematics},
    VOLUME = {15},
      YEAR = {2015},
    NUMBER = {3},
     PAGES = {681--713},
}

@article {Gaspar,
    AUTHOR = {Gaspar, Jaime},
     TITLE = {All triangles at once},
   JOURNAL = {Amer. Math. Monthly},
  FJOURNAL = {American Mathematical Monthly},
    VOLUME = {122},
      YEAR = {2015},
    NUMBER = {10},
     PAGES = {982},
}

@article {Ghosh,
    AUTHOR = {Ghosh, Birendranath},
     TITLE = {Random distances within a rectangle and between two
              rectangles},
   JOURNAL = {Bull. Calcutta Math. Soc.},
  FJOURNAL = {Bulletin of the Calcutta Mathematical Society},
    VOLUME = {43},
      YEAR = {1951},
     PAGES = {17--24},
}

@article{Boehm,
 URL = {http://www.jstor.org/stable/2132885},
 author = {Cecil C. Rousseau and Otto G. Ruehr},
 journal = {SIAM Review},
 number = {2},
 pages = {313--332},
 publisher = {Society for Industrial and Applied Mathematics},
 title = {Problems and Solutions},
 volume = {38},
 year = {1996}
}

@misc{Gausscircle-wiki,
    author = {Wikipedia-contributors},
    title = {Gauss circle problem --- {Wikipedia}{,} The Free Encyclopedia},
    note = {\url{https://en.wikipedia.org/w/index.php?title=Gauss_circle_problem&oldid=1300197383}},
  }

@misc{Weyl-wiki,
    author = {Wikipedia-contributors},
    title = {Equidistribution theorem --- {Wikipedia}{,} The Free Encyclopedia},
    note = {\url{https://en.wikipedia.org/w/index.php?title=Equidistribution_theorem&oldid=1306115791}},
  }


\end{document}